\documentclass[12pt]{amsart}
\usepackage{amsmath,amsthm,amsfonts,amssymb,amscd,euscript}
\usepackage{color}

\newlength{\defbaselineskip}
\theoremstyle{plain}
\newtheorem{theorem}{Theorem}[section]

\theoremstyle{definition}

\newtheorem{conjecture}[theorem]{Conjecture}

\theoremstyle{plain}
\newtheorem{thm}{Theorem}

\newtheorem{lem}[thm]{Lemma}
\newtheorem{cor}[thm]{Corollary}

\newtheorem{conj}[thm]{Conjecture}

\theoremstyle{definition}
\newtheorem{defn}[thm]{Definition}

\newtheorem{exmp}[thm]{Example}

\newtheorem{rem}[thm]{Remark}
\numberwithin{equation}{section}

\newcommand{\bigboxs}[1]
{ \multiput(#1)(60,0){2}
 {\line(0,60){60}}
\multiput(#1)(0,60){2}
 {\line(60,0){60}}
}
\newcommand{\boxs}[1]
{ \multiput(#1)(20,0){2}
 {\line(0,20){20}}
\multiput(#1)(0,20){2}
 {\line(20,0){20}}
}

\begin{document}

\title[Proof of a positivity conjecture of M. Kontsevich]{Proof of a positivity conjecture of M. Kontsevich on non-commutative cluster variables}
\author{Kyungyong Lee and Ralf Schiffler}
\thanks{Research of K.L. is partially supported by NSF grant DMS 0901367. Research of R.S. is partially supported by NSF grant DMS 1001637. }

\address{Department of Mathematics, Wayne State University, Detroit, MI 48202}
\email{{\tt klee@math.wayne.edu}}
\address{Department of Mathematics, University of Connecticut, Storrs, CT 06269}
\email{{\tt schiffler@math.uconn.edu}}



 \maketitle
 



\section{introduction}

Let $K = k(x,y)$ be the skew field of rational functions in the non-commutative
variables $x$ and $y$, where the ground field $k$ is $\mathbb{Q}$ or any field containing $\mathbb{Q}$, for example $\mathbb{Q}(q)$. For any positive integer $r$, let $F_r$ be the Kontsevich automorphism of $K$, which is defined by
$$F_r(\lambda)=\lambda,\text{ for all }\lambda\in k\text{ and}$$
\begin{equation}\label{Kont_map}F_r : \left\{\begin{array}{l}x \mapsto xyx^{-1} \\ y \mapsto (1+y^r)x^{-1}.\end{array}\right.\end{equation}
 
The main result of this paper is the proof of a special case of the following conjecture.
\begin{conjecture}[Kontsevich] \label{conjintro} For all positive integers $r_1, r_2$ and  for all $m\ge 0$, the expressions
 $$\left( F_{r_2}\circ F_{r_1}\right)^m (x) \qquad\textup{ and }\qquad\left( F_{r_2}\circ F_{r_1}\right)^m (y)$$ are non-commutative Laurent polynomials in $x$ and $y$ with non-negative integer coefficients. 
\end{conjecture}

We shall prove the conjecture in the case $r_1=r_2$ by providing an explicit combinatorial formula for these expressions as a sum over certain sets of lattice paths $\beta$, where each summand is a Laurent monomial given by the weight of the paths in $\beta$. As a direct consequence of this formula, we have the following.

\begin{theorem}
Conjecture \ref{conjintro} holds whenever $r_1=r_2$.
\end{theorem}

 Let us point out that, if the variables $x$ and $y$ were commutative variables, then the automorphism $F_r$ would describe precisely the exchange relations for the mutations in a skew-symmetric cluster algebra $\mathcal{A}_r$  of rank 2, and  our above mentioned formula is a non-commutative version of a formula for the cluster variables in $\mathcal{A}_r$ which we obtained earlier, see \cite{LS}.

In the special cases where $(r_1,r_2)=(2,2),(4,1),(1,4)$ the conjecture has been shown by DiFrancesco and Kedem in \cite{DK}. Moreover, it has been shown that the expressions in Conjecture \ref{conjintro} are Laurent polynomials for any choice of $(r_1,r_2)$ by Berenstein and Retakh \cite{BR} and earlier by Usnich \cite {U} in the case $r_1=r_2$.

\noindent \emph{Acknowledgements.} We are grateful to Philippe Di Francesco, Gr\'egoire Dupont, Sergey Fomin, David Hernandez, Rinat Kedem, Bernhard Keller, Maxim Kontsevich, Rob Lazarsfeld, Gregg Musiker, Vladimir Retakh, Dylan Rupel, Hugh Thomas, and Andrei Zelevinsky for their valuable advice, suggestions, comments, discussions and correspondence. 

\section{Main Result}

Fix a positive integer $r\geq 2$. 

\begin{defn}\label{cn}
Let $\{c_n\}$ be the sequence  defined by the recurrence relation $$c_n=rc_{n-1} -c_{n-2},$$ with the initial condition $c_1=0$, $c_2=1$. When $r=2$, $c_n=n-1$. When $r>2$, it is easy to see that 
$$\aligned
c_n&= \frac{1}{\sqrt{r^2-4}  }\left(\frac{r+\sqrt{r^2-4}}{2}\right)^{n-1} - \frac{1}{\sqrt{r^2-4}  }\left(\frac{r-\sqrt{r^2-4}}{2}\right)^{n-1}\\ &= \sum_{i\geq 0} (-1)^i { {n-2-i} \choose i }r^{n-2-2i}.
\endaligned$$ For example, for $r=3$, the sequence $c_n$ takes the following values:
\[0,1,3,8,21,55,144,...\]
\end{defn}

In order to state our theorem, we fix an integer $n\geq 4$. Consider a rectangle with vertices $(0,0),(0,c_{n-2}),(c_{n-1}-c_{n-2},c_{n-2})$ and $(c_{n-1}-c_{n-2},0)$. In what follows, by the diagonal we mean the line segment from $(0,0)$ to $(c_{n-1}-c_{n-2},c_{n-2})$. A Dyck path is a lattice path
from $(0, 0)$  to $(c_{n-1}-c_{n-2},c_{n-2})$ that proceeds by NORTH or EAST steps and
never goes above the diagonal. 

\begin{defn}
A Dyck path below the diagonal is said to be maximal if no subpath of any other Dyck path lies above it. The maximal Dyck path, denoted by $\mathcal{D}_n$, consists of $(w_0, \alpha_1,w_1,\cdots, \alpha_{c_{n-1}}, w_{c_{n-1}})$, where $w_0,\cdots,w_{c_{n-1}}$ are vertices and $\alpha_1,\cdots, \alpha_{c_{n-1}}$ are edges, such that $w_0=(0,0)$ is the south-west corner of the rectangle, $\alpha_i $ connects $w_{i-1}$ and $w_i$, and $w_{c_{n-1}}=(c_{n-1}-c_{n-2},c_{n-2})$ is the north-east corner of the rectangle.
\end{defn}

\begin{rem}
The word obtained from $\mathcal{D}_n$ by forgetting the vertices $w_i$ and replacing each horizontal edge by the letter $x$ and each vertical edge by the letter $y$ is (by definition) the Christoffel word of slope $c_{n-2}/(c_{n-1}-c_{n-2})$.
\end{rem}

\begin{exmp}Let $r=3$ and $n=5$. Then $\mathcal{D}_5$ is illustrated as follows.
$$
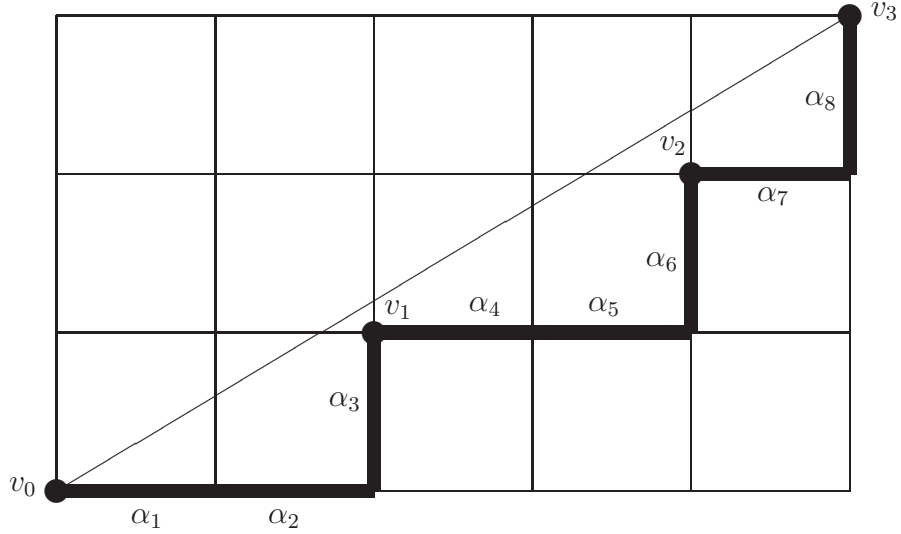
\hspace{16pt} \begin{picture}(300,190)
\bigboxs{0,0}\bigboxs{60,0}\bigboxs{120,0}\bigboxs{180,0}\bigboxs{240,0}
\bigboxs{0,60}\bigboxs{60,60}\bigboxs{120,60}\bigboxs{180,60}\bigboxs{240,60}
\bigboxs{0,120}\bigboxs{60,120}\bigboxs{120,120}\bigboxs{180,120}\bigboxs{240,120}
\put(28,-12){$\tiny{\alpha_1}$}\put(80,-12){$\tiny{\alpha_2}$}
\put(201,68){$\tiny{\alpha_5}$}\put(156,68){$\tiny{\alpha_4}$}
\put(265,110){$\tiny{\alpha_7}$}
\put(103,32){$\tiny{\alpha_3}$}\put(223,85){$\tiny{\alpha_6}$}\put(283,146){$\tiny{\alpha_8}$}
\linethickness{1pt}\put(0,0){\line(5,3){300}}
\linethickness{5pt}\put(0,0){\line(1,0){120}}
\linethickness{5pt}\put(120,0){\line(0,1){60}}
\linethickness{5pt}\put(120,60){\line(1,0){120}}
\linethickness{5pt}\put(240,60){\line(0,1){60}}
\linethickness{5pt}\put(240,120){\line(1,0){60}}
\linethickness{5pt}\put(300,120){\line(0,1){60}}
\put(0,0){\circle*{9}}\put(120,60){\circle*{9}}\put(240,120){\circle*{9}}\put(300,180){\circle*{9}}\put(-18,0){$v_0$}\put(124,67){$v_1$}\put(228,129){$v_2$}\put(308,180){$v_3$}
\end{picture}$$
\end{exmp}

\bigskip
\begin{defn}
Let $v_i$ be the upper end point of the $i$-th vertical edge of $\mathcal{D}_n$. More precisely, let $i_1<\cdots<i_{c_{n-2}}$ be the sequence of integers such that $\alpha_{i_j}$ is vertical for any $1\leq j\leq c_{n-2}$. Define a sequence $v_0,v_1,\cdots,v_{c_{n-2}}$ of vertices by $v_0=(0,0)$ and $v_j=w_{i_j}$. 
\end{defn}

We introduce certain special subpaths called colored subpaths. These colored subpaths are defined by certain slope conditions as follows.

\begin{defn}
For any $i<j$, let $s_{i,j}$ be the slope of the line through $v_i$ and $v_j$. Let $s$ be the slope of the diagonal, that is, $s=s_{0,c_{n-2}}$. 
\end{defn}

\begin{defn}[Colored subpaths]\label{alpha(i,k)}
For any $0\leq i<k\leq c_{n-2}$, let $\alpha(i,k)$ be the subpath of $\mathcal{D}_n$ defined as follows (for illustrations see Example~\ref{mainexmp}).

\noindent (1) If $s_{i,t}\leq s$ for all $t$ such that $i<t\leq k$, then let $\alpha(i,k)$ be the subpath from $v_i$ to $v_k$. Each of these subpaths will be called a BLUE subpath. See Example~\ref{mainexmp}.

\noindent (2) If $s_{i,t}> s$ for some $i<t\leq k$, then

(2-a) if the smallest such $t$ is of the form $i+c_m-wc_{m-1}$ for some integers $3\le m\le n-1$ and $1\leq w< r-1$, then let $\alpha(i,k)$ be the subpath from $v_i$ to $v_k$. Each of these subpaths will be called a GREEN subpath. When $m$ and $w$ are specified, it will be said to be $(m,w)$-green.

(2-b) otherwise, let $\alpha(i,k)$ be the subpath from the immediate predecessor of $v_i$ to $v_k$. Each of these subpaths will be called a RED subpath.
\end{defn}

Note that every pair $(i,k)$ defines exactly one subpath $\alpha(i,k)$. We call these subpaths the \emph{colored subpaths} of $\mathcal{D}_n$. We denote the set of all these subpaths together with the single edges $\alpha_i$ by $\mathcal{P}(\mathcal{D}_n)$, that is,$$\mathcal{P}(\mathcal{D}_n)=\{\alpha(i,k)\,|\, 0\leq i<k\leq c_{n-2}\} \cup \{\alpha_1,\cdots,\alpha_{c_{n-1}} \}.$$

Now we define a set $\mathcal{F}(\mathcal{D}_n)$ of certain sequences of non-overlapping subpaths of $\mathcal{D}_n$. This set will parametrize the monomials in our expansion formula.

\begin{defn}
Let $$ \mathcal{F}(\mathcal{D}_n)=\left\{ \{\beta_1,\cdots,\beta_t\}\,\left|\aligned &\bullet \,t\geq 0,\ \beta_j\in \mathcal{P}(\mathcal{D}_n)\text{ for all }1\leq j\leq t,\,\\ &\bullet \text{ if }j\neq j'\text{ then }\beta_j\text{ and }\beta_{j'}\text{ have no common edge,}\\ &\bullet \text{ if }\beta_j=\alpha(i,k)\text{ and }\beta_{j'}=\alpha(i',k')\text{ then }i\neq k'\text{ and }i'\neq k,\\
&\bullet \text{ and if }\beta_j\text{ is }(m,w)\text{-green then at least one of the }(c_{m-1}-wc_{m-2})\\&\,\,\,\,\,\,\text{ preceding edges of }v_i\text{ is contained in some }\beta_{j'} \endaligned \right.\right\}. $$ 
\end{defn}

For each $\beta\in \mathcal{F}(\mathcal{D}_n)$, we say that $\alpha_i$ is \emph{supported on} $\beta$ if and only if $\alpha_i\in\beta$ or $\alpha_i$ is contained in some blue, green or red subpath $\beta_j\in\beta$. The \emph{support} of $\beta$, denoted by $\text{supp}(\beta)$, is defined to be the union of $\alpha_i$'s that are supported on $\beta$.

\begin{defn}
For each $\beta\in \mathcal{F}(\mathcal{D}_n)$ and each $i\in\{1,\cdots,c_{n-1}\}$, let
$$\beta_{[i]}=\left\{\begin{array}{ll}
x^{-1}y^{r}, &\text{ if }\alpha_i\text{ is not supported on }\beta\text{ and }\alpha_i\text{ is horizontal;}\\
&\\
x^{-1}y^{r-1}, &\text{ if }\alpha_i\text{ is not supported on }\beta\text{ and }\alpha_i\text{ is vertical;}\\ 
&\\
x^{-1}y^{0}, &\text{ if }\alpha_i\in\beta\text{ and }\alpha_i\text{ is horizontal;}\\
&\\
x^{-1}y^{-1}, &\text{ if }\alpha_i\in\beta\text{ and }\alpha_i\text{ is vertical;}\\
&\\
x^{0}y^{0},  &\text{ if }\alpha_i\text{ is horizontal and }\alpha_i\in\alpha(j,k)\in\beta\text{ for 
some }j,k;\\
&\\
x^{0}y^{-1}, &\text{ if }\alpha_i\text{ is vertical, }\alpha_{i-r+1}\text{ is horizontal, and }\alpha_i,\alpha_{i-r+1}\in\alpha(j,k)\in\beta\text{ for some }j,k;\\ 
&\\
x^{1}y^{-1}, &\text{ if }\alpha_i\text{ and }\alpha_{i-r+1}\text{ are  vertical, and }\alpha_i,\alpha_{i-r+1}\in\alpha(j,k)\in\beta\text{ for some }j,k;\\ 
&\\
x^{-1}y^{-1}, &\text{ if }\alpha_i\text{ is the first (vertical) edge of a red subpath }\alpha(j,k)\text{ in }\beta. 
\end{array}\right.$$
\end{defn}

Note that the last three cases exhaust all possibilities for $\alpha_i$ being a vertical edge contained in some $\alpha(j,k)$ in $\beta$, because if in addition $\alpha_{i-r+1}\notin \alpha(j,k)$ then $\alpha_i$ must be the first vertical edge of a red subpath.


Recall from the introduction that the Kontsevich automorphism $F_r$ is given by 
\begin{equation}\label{Kont_map}F_r : \left\{\begin{array}{l}x \mapsto xyx^{-1} \\ y \mapsto (1+y^r)x^{-1}.\end{array}\right.\end{equation}
Let $F^{-1}_r$ be the inverse of $F_r$, namely,
\begin{equation}\label{Kont_inverse_map}F^{-1}_r : \left\{\begin{array}{l}x \mapsto (1+x^r)y^{-1} \\ y \mapsto yxy^{-1}.\end{array}\right.\end{equation}

Consider a sequence $\{r_n\}_{n\in \mathbb{Z}}$ of positive integers. For any positive integer $n$, let
$$x_n=(F_{r_n}\circ\cdots \circ F_{r_2} \circ F_{r_1})(x)=F_{r_n}(\cdots F_{r_2}(F_{r_1}(x))\cdots)\text{ and } y_n =(F_{r_n}\circ\cdots \circ F_{r_2} \circ F_{r_1})(y),$$
and let
$$x_{-n}=(F^{-1}_{r_{-n+1}}\circ\cdots \circ F^{-1}_{r_{-1}} \circ F^{-1}_{r_0})(x)\text{ and } y_{-n}=(F^{-1}_{r_{-n+1}}\circ\cdots \circ F^{-1}_{r_{-1}} \circ F^{-1}_{r_0})(y).$$
Let $x_0=x$ and $y_0=y$.

\begin{conj}[M. Kontsevich]\label{Kontconj}
Let $r_1$ and $r_2$ be arbitrary positive integers. Assume that $r_{2i+1}=r_1$ and $r_{2i}=r_2$ for every $i\in \mathbb{Z}$. Then, for any integer $n$, both $x_n$ and $y_n$ are non-commutative Laurent polynomials of $x$ and $y$ with non-negative integer coefficients.
\end{conj}

We are now ready to state our main result. For monomials $A_i$ in $K$, we let $\prod_{i=1}^m A_i$ denote the non-commutative product $A_1A_2\cdots A_m$.

\begin{thm}\label{mainthm}
If $r_n=r$ for all $n$ then for $n\geq 4$,
\begin{equation}\label{maineq1}x_{n-1}=\sum_{\mathbf{\beta}\in\mathcal{F}(\mathcal{D}_n)} xyx^{-1}y^{-1}x\left(\prod_{i=1}^{c_{n-1}}\beta_{[i]}\right)x^{-1}. \end{equation}
\end{thm}
\begin{cor} Conjecture~\ref{Kontconj} holds in the case $r_1=r_2$.
\end{cor}
\begin{proof} The theorem implies that $x_n$ is a non-commutative Laurent polynomial of $x$ and $y$ with non-negative integer coefficients, for $n\ge 0$. The statements for $x_n (n<0)$ and $y_n$ then follow from a symmetry argument, see \cite[section 2.3]{DK} or \cite[Lemma 7]{BR}.
\end{proof}

\begin{rem}\label{rem1}
The right hand side of equation (\ref{maineq1}) can be written as a double sum as follows.
\begin{equation}\label{doublesum}x_{n-1}=\sum_{i_j,k_j} \sum_\beta  xyx^{-1}y^{-1}x\left(\prod_{i=1}^{c_{n-1}}\beta_{[i]}\right)x^{-1}, \end{equation}
where the first sum is over all sequences $0\le i_1<k_1<\cdots<i_\ell<k_\ell\le c_{n-2}$ and the second sum is over all $\beta\in \mathcal{F}(\mathcal{D}_n)$ whose colored subpaths are precisely the $\alpha(i_j,k_j)$ for $1\le j\le \ell$.
\end{rem}

\begin{exmp}\label{mainexmp}  Let $r=3$ and $n=5$. 
We use the following presentation for monomials in $K$:
$$x^{a_1}y^{b_1}x^{a_2}y^{b_2}\cdots  x^{a_{m-1}}y^{b_{m-1}} x^{a_m}y^{b_m}\longleftrightarrow \left(\begin{array}{ccccc} a_1 & a_2 & \cdots & a_{m-1}& a_m\\  b_1 & b_2 & \cdots & b_{m-1} & b_m \end{array} \right).$$
These expressions are not necessarily minimal, i.e. some of $a_i$ or $b_i$ are allowed to be zeroes.

The illustrations below show the possible configurations for $\beta\in\mathcal{F}(\mathcal{D}_n)$.  If the edge $\alpha_i$ is marked $\begin{picture}(25,10)\linethickness{3pt}\put(2,2){\line(1,0){20}}\color{white}\put(9,2){\line(1,0){6}} \end{picture}$, then $\alpha_i$ can occur in $\beta$. Using the double sum expression of equation (\ref{doublesum}), we get that $x_{n-1}$ is the sum of all the sums below.
$$
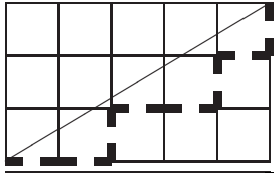
\begin{array}{cc}
&\\
&\\
&\\
\hspace{16pt} \begin{picture}(100,60)
\boxs{0,0}\boxs{20,0}\boxs{40,0}\boxs{60,0}\boxs{80,0}
\boxs{0,20}\boxs{20,20}\boxs{40,20}\boxs{60,20}\boxs{80,20}
\boxs{0,40}\boxs{20,40}\boxs{40,40}\boxs{60,40}\boxs{80,40}
\linethickness{1pt}\put(0,0){\line(5,3){100}}
\linethickness{3pt}\put(0,0){\line(1,0){40}}
\linethickness{3pt}\put(40,0){\line(0,1){20}}
\linethickness{3pt}\put(40,20){\line(1,0){40}}
\linethickness{3pt}\put(80,20){\line(0,1){20}}
\linethickness{3pt}\put(80,40){\line(1,0){20}}
\linethickness{3pt}\put(100,40){\line(0,1){20}}\color{white}\put(47,20){\line(1,0){6}}\linethickness{3pt}\color{white}\put(67,20){\line(1,0){6}}\color{white}\put(7,0){\line(1,0){6}}\linethickness{3pt}\color{white}\put(27,0){\line(1,0){6}}\color{white}\put(87,40){\line(1,0){6}}\put(40,7){\line(0,1){6}}\put(80,27){\line(0,1){6}}\put(100,47){\line(0,1){6}}
\end{picture}
&\,\,\,\,\,\,\,\,\,\,\,\,\,\,\,\,\,\,\,\,\,\,\,\,\,\,\,\,\,\,\,\,\,\,\,\,\,\,\,\,\,\,\,\,\,\,\,\,\,\,\,\,\,\begin{picture}(300,60)\put(-80,35){$\tiny{\sum_{\begin{array}{l}\mathbf{\beta}\subset \{\alpha_1,\cdots,\alpha_{8} \}  \end{array}}xyx^{-1}y^{-1}x\left(\prod_{i=1}^{c_{n-1}}\beta_{[i]}\right)x^{-1}}$}\put(-80,15){$= A_1$}\end{picture}
\\
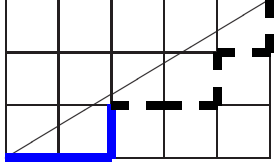
\hspace{16pt} \begin{picture}(100,60)
\boxs{0,0}\boxs{20,0}\boxs{40,0}\boxs{60,0}\boxs{80,0}
\boxs{0,20}\boxs{20,20}\boxs{40,20}\boxs{60,20}\boxs{80,20}
\boxs{0,40}\boxs{20,40}\boxs{40,40}\boxs{60,40}\boxs{80,40}
\linethickness{1pt}\put(0,0){\line(5,3){100}}
\linethickness{3pt}\color{blue}\put(0,0){\line(1,0){40}}
\linethickness{3pt}\color{blue}\put(40,0){\line(0,1){20}}
\linethickness{3pt}\color{black}\put(40,20){\line(1,0){40}}
\linethickness{3pt}\put(80,20){\line(0,1){20}}
\linethickness{3pt}\put(80,40){\line(1,0){20}}
\linethickness{3pt}\put(100,40){\line(0,1){20}}\color{white}\put(47,20){\line(1,0){6}}\color{white}\put(67,20){\line(1,0){6}}\put(87,40){\line(1,0){6}}\put(80,27){\line(0,1){6}}\put(100,47){\line(0,1){6}}
\end{picture}
&\,\,\,\,\,\,\,\,\,\,\,\,\,\,\,\,\,\,\,\,\,\,\,\,\,\,\,\,\,\,\,\,\,\,\,\,\,\,\,\,\,\,\,\,\,\,\,\,\,\,\,\,\,\begin{picture}(300,60)\put(-80,35){$\tiny{\sum_{\begin{array}{l}\{\alpha(0,1)\}\subset\mathbf{\beta}\subset\{\alpha(0,1)\}\cup  \{\alpha_4,\cdots,\alpha_{8} \}  \end{array}}xyx^{-1}y^{-1}x\left(\prod_{i=1}^{c_{n-1}}\beta_{[i]}\right)x^{-1}}$}\put(-80,15){$= A_2$}\end{picture}\\
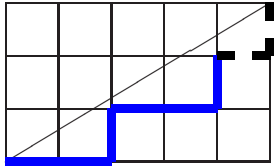
\hspace{16pt} \begin{picture}(100,60)
\boxs{0,0}\boxs{20,0}\boxs{40,0}\boxs{60,0}\boxs{80,0}
\boxs{0,20}\boxs{20,20}\boxs{40,20}\boxs{60,20}\boxs{80,20}
\boxs{0,40}\boxs{20,40}\boxs{40,40}\boxs{60,40}\boxs{80,40}
\linethickness{1pt}\put(0,0){\line(5,3){100}}
\linethickness{3pt}\color{blue}\put(0,0){\line(1,0){40}}
\linethickness{3pt}\color{blue}\put(40,0){\line(0,1){20}}
\linethickness{3pt}\put(40,20){\line(1,0){40}}
\linethickness{3pt}\put(80,20){\line(0,1){20}}
\linethickness{3pt}\color{black}\put(80,40){\line(1,0){20}}
\linethickness{3pt}\put(100,40){\line(0,1){20}}\color{white}\put(87,40){\line(1,0){6}}\put(100,47){\line(0,1){6}}
\end{picture}
&\,\,\,\,\,\,\,\,\,\,\,\,\,\,\,\,\,\,\,\,\,\,\,\,\,\,\,\,\,\,\,\,\,\,\,\,\,\,\,\,\,\,\,\,\,\,\,\,\,\,\,\,\,\begin{picture}(300,60)\put(-80,35){$\tiny{\sum_{\begin{array}{l}\{\alpha(0,2)\}\subset\mathbf{\beta}\subset\{\alpha(0,2)\}\cup  \{\alpha_7,\alpha_{8} \}  \end{array}}xyx^{-1}y^{-1}x\left(\prod_{i=1}^{c_{n-1}}\beta_{[i]}\right)x^{-1}}$}\put(-80,15){$= A_3$}\end{picture}\\
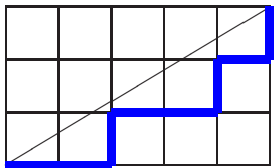
\hspace{16pt} \begin{picture}(100,60)
\boxs{0,0}\boxs{20,0}\boxs{40,0}\boxs{60,0}\boxs{80,0}
\boxs{0,20}\boxs{20,20}\boxs{40,20}\boxs{60,20}\boxs{80,20}
\boxs{0,40}\boxs{20,40}\boxs{40,40}\boxs{60,40}\boxs{80,40}
\linethickness{1pt}\put(0,0){\line(5,3){100}}
\linethickness{3pt}\color{blue}\put(0,0){\line(1,0){40}}
\linethickness{3pt}\color{blue}\put(40,0){\line(0,1){20}}
\linethickness{3pt}\put(40,20){\line(1,0){40}}
\linethickness{3pt}\put(80,20){\line(0,1){20}}
\linethickness{3pt}\put(80,40){\line(1,0){20}}
\linethickness{3pt}\put(100,40){\line(0,1){20}}
\end{picture}
&\,\,\,\,\,\,\,\,\,\,\,\,\,\,\,\,\,\,\,\,\,\,\,\,\,\,\,\,\,\,\,\,\,\,\,\,\,\,\,\,\,\,\,\,\,\,\,\,\,\,\,\,\,\begin{picture}(300,60)\put(-80,35){$\tiny{\sum_{\begin{array}{l}\mathbf{\beta}=\{\alpha(0,3)\} \end{array}}xyx^{-1}y^{-1}x\left(\prod_{i=1}^{c_{n-1}}\beta_{[i]}\right)x^{-1}}$}\put(-80,15){$= A_4$}\end{picture}\\
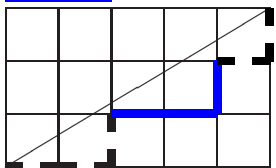
\hspace{16pt} \begin{picture}(100,60)
\boxs{0,0}\boxs{20,0}\boxs{40,0}\boxs{60,0}\boxs{80,0}
\boxs{0,20}\boxs{20,20}\boxs{40,20}\boxs{60,20}\boxs{80,20}
\boxs{0,40}\boxs{20,40}\boxs{40,40}\boxs{60,40}\boxs{80,40}
\linethickness{1pt}\put(0,0){\line(5,3){100}}
\linethickness{3pt}\put(0,0){\line(1,0){40}}
\linethickness{3pt}\put(40,0){\line(0,1){20}}
\linethickness{3pt}\color{blue}\put(40,20){\line(1,0){40}}
\linethickness{3pt}\put(80,20){\line(0,1){20}}
\linethickness{3pt}\color{black}\put(80,40){\line(1,0){20}}
\linethickness{3pt}\put(100,40){\line(0,1){20}}\color{white}\put(87,40){\line(1,0){6}}\put(7,0){\line(1,0){6}}\put(27,0){\line(1,0){6}}\put(40,7){\line(0,1){6}}\put(100,47){\line(0,1){6}}
\end{picture}
&\,\,\,\,\,\,\,\,\,\,\,\,\,\,\,\,\,\,\,\,\,\,\,\,\,\,\,\,\,\,\,\,\,\,\,\,\,\,\,\,\,\,\,\,\,\,\,\,\,\,\,\,\,\begin{picture}(300,60)\put(-80,35){$\tiny{\sum_{\begin{array}{l}\{\alpha(1,2)\}\subset \mathbf{\beta}\subset\{\alpha(1,2)\}\cup  \{\alpha_1,\alpha_2,\alpha_{3},\alpha_{7},\alpha_{8} \}  \end{array}}xyx^{-1}y^{-1}x\left(\prod_{i=1}^{c_{n-1}}\beta_{[i]}\right)x^{-1}}$}\put(-80,15){$= A_5$}\end{picture}\\
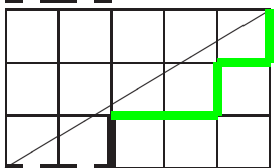
\hspace{16pt} \begin{picture}(100,60)
\boxs{0,0}\boxs{20,0}\boxs{40,0}\boxs{60,0}\boxs{80,0}
\boxs{0,20}\boxs{20,20}\boxs{40,20}\boxs{60,20}\boxs{80,20}
\boxs{0,40}\boxs{20,40}\boxs{40,40}\boxs{60,40}\boxs{80,40}
\linethickness{1pt}\put(0,0){\line(5,3){100}}
\linethickness{3pt}\put(0,0){\line(1,0){40}}
\linethickness{3pt}\put(40,0){\line(0,1){20}}
\linethickness{3pt}\color{green}\put(40,20){\line(1,0){40}}
\linethickness{3pt}\put(80,20){\line(0,1){20}}
\linethickness{3pt}\put(80,40){\line(1,0){20}}
\linethickness{3pt}\put(100,40){\line(0,1){20}}\color{white}\put(7,0){\line(1,0){6}}\put(27,0){\line(1,0){6}}
\end{picture}
&\,\,\,\,\,\,\,\,\,\,\,\,\,\,\,\,\,\,\,\,\,\,\,\,\,\,\,\,\,\,\,\,\,\,\,\,\,\,\,\,\,\,\,\,\,\,\,\,\,\,\,\,\,\begin{picture}(300,60)\put(-80,35){$\tiny{\sum_{\begin{array}{l}\{\alpha(1,3)\}\cup  \{\alpha_3\}\subset\mathbf{\beta}\subset\{\alpha(1,3)\}\cup  \{\alpha_1,\alpha_2,\alpha_3 \}  \end{array}}xyx^{-1}y^{-1}x\left(\prod_{i=1}^{c_{n-1}}\beta_{[i]}\right)x^{-1}}$}\put(-80,15){$= A_6$}\end{picture}\\
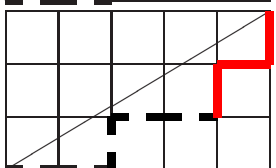
\hspace{16pt} \begin{picture}(100,60)
\boxs{0,0}\boxs{20,0}\boxs{40,0}\boxs{60,0}\boxs{80,0}
\boxs{0,20}\boxs{20,20}\boxs{40,20}\boxs{60,20}\boxs{80,20}
\boxs{0,40}\boxs{20,40}\boxs{40,40}\boxs{60,40}\boxs{80,40}
\linethickness{1pt}\put(0,0){\line(5,3){100}}\linethickness{3pt}
\linethickness{3pt}\put(0,0){\line(1,0){40}}\color{white}\put(7,0){\line(1,0){6}}\linethickness{3pt}\color{white}\put(27,0){\line(1,0){6}}
\linethickness{3pt}\color{black}\put(40,0){\line(0,1){20}}
\linethickness{3pt}\put(40,20){\line(1,0){40}}\color{white}\put(47,20){\line(1,0){6}}\linethickness{3pt}\color{white}\put(67,20){\line(1,0){6}}
\linethickness{3pt}\color{red}\put(80,20){\line(0,1){20}}
\linethickness{3pt}\put(80,40){\line(1,0){20}}
\linethickness{3pt}\put(100,40){\line(0,1){20}}\color{white}\put(40,7){\line(0,1){6}}
\end{picture}
&\,\,\,\,\,\,\,\,\,\,\,\,\,\,\,\,\,\,\,\,\,\,\,\,\,\,\,\,\,\,\,\,\,\,\,\,\,\,\,\,\,\,\,\,\,\,\,\,\,\,\,\,\,\begin{picture}(300,60)\put(-80,35){$\tiny{\sum_{\begin{array}{l} \{\alpha(2,3)\}\subset\mathbf{\beta}\subset\{\alpha(2,3)\}\cup  \{\alpha_1,\cdots,\alpha_{5} \}  \end{array}}xyx^{-1}y^{-1}x\left(\prod_{i=1}^{c_{n-1}}\beta_{[i]}\right)x^{-1}}$}\put(-80,15){$= A_7$}\end{picture}\\
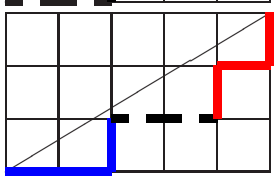
\hspace{16pt} \begin{picture}(100,60)
\boxs{0,0}\boxs{20,0}\boxs{40,0}\boxs{60,0}\boxs{80,0}
\boxs{0,20}\boxs{20,20}\boxs{40,20}\boxs{60,20}\boxs{80,20}
\boxs{0,40}\boxs{20,40}\boxs{40,40}\boxs{60,40}\boxs{80,40}
\linethickness{1pt}\put(0,0){\line(5,3){100}}
\linethickness{3pt}\color{blue}\put(0,0){\line(1,0){40}}
\linethickness{3pt}\put(40,0){\line(0,1){20}}
\linethickness{3pt}\color{black}\put(40,20){\line(1,0){40}}
\linethickness{3pt}\color{white}\put(47,20){\line(1,0){6}}\linethickness{3pt}\color{white}\put(67,20){\line(1,0){6}}
\linethickness{3pt}\color{red}\put(80,20){\line(0,1){20}}
\linethickness{3pt}\put(80,40){\line(1,0){20}}
\linethickness{3pt}\put(100,40){\line(0,1){20}}
\end{picture}
&\,\,\,\,\,\,\,\,\,\,\,\,\,\,\,\,\,\,\,\,\,\,\,\,\,\,\,\,\,\,\,\,\,\,\,\,\,\,\,\,\,\,\,\,\,\,\,\,\,\,\,\,\,\begin{picture}(300,60)\put(-80,35){$\tiny{\sum_{\begin{array}{l} \{\alpha(0,1),\alpha(2,3)\}\subset\mathbf{\beta}\subset\{\alpha(0,1),\alpha(2,3)\}\cup  \{\alpha_4,\alpha_{5} \}  \end{array}}xyx^{-1}y^{-1}x\left(\prod_{i=1}^{c_{n-1}}\beta_{[i]}\right)x^{-1}}$}\put(-80,15){$= A_8,$}\end{picture}\end{array}
$$

\clearpage

\noindent where
\tiny{
$$\begin{array}{ll} A_1=\sum_{\delta_1,\delta_2,...,\delta_8\in\{0,1\}}&\left(\begin{array}{ccccccccccc} 1 &  -1 & -1&-1& -1&-1&-1& -1&-1&-1\\ 1&2-3\delta_1 &3-3\delta_2&2-3\delta_3& 3-3\delta_4&3-3\delta_5&2-3\delta_6&3-3\delta_7&2-3\delta_8& 0 \end{array} \right),\\
A_2=\sum_{\delta_4,\delta_5,...,\delta_8\in\{0,1\}}&
\left(\begin{array}{ccccccccccc}  1 &  -1 & 0&1&-1& -1&-1&-1&-1& -1\\
1& \hspace{8pt} -1\hspace{8pt}  &\hspace{11.3pt} 0\hspace{11.3pt}&\hspace{8pt} -1\hspace{8pt}& 3-3\delta_4&3-3\delta_5&2-3\delta_6&3-3\delta_7&2-3\delta_8& 0 \end{array} \right),\\
A_3=\sum_{\delta_7,\delta_8\in\{0,1\}}&\left(\begin{array}{ccccccccccc}  1 &  -1 & 0&1&0& 0&0&-1& -1&-1\\ 1&\hspace{8pt} -1\hspace{8pt} &\hspace{11.3pt} 0\hspace{11.3pt}&\hspace{8pt} -1\hspace{8pt}& \hspace{11.3pt} 0\hspace{11.3pt}&\hspace{11.3pt} 0\hspace{11.3pt}&\hspace{8pt} -1\hspace{8pt}&3-3\delta_7&2-3\delta_8& 0\end{array} \right),\\
A_4=&\left(\begin{array}{ccccccccccc}  1 &  -1 & 0&1& 0&0&0& 0&1&-1\\ 1&\hspace{8pt} -1\hspace{8pt} &\hspace{11.3pt} 0\hspace{11.3pt}&\hspace{8pt} -1\hspace{8pt}& \hspace{11.3pt} 0\hspace{11.3pt}&\hspace{11.3pt} 0\hspace{11.3pt}&\hspace{8pt} -1\hspace{8pt}&\hspace{11.3pt} 0\hspace{11.3pt}&\hspace{8pt} -1\hspace{8pt}& 0 \end{array} \right)
,\\
A_5=\sum_{\delta_1,\delta_2,\delta_3,\delta_7,\delta_8\in\{0,1\}}&\left(\begin{array}{ccccccccccc}  1 &  -1 & -1&-1&0& 0&0&-1& -1&-1\\ 1&2-3\delta_1 &3-3\delta_2&2-3\delta_3& \hspace{11.3pt} 0\hspace{11.3pt}&\hspace{11.3pt} 0\hspace{11.3pt}&\hspace{8pt} -1\hspace{8pt}&3-3\delta_7&2-3\delta_8& 0\end{array} \right),\\
A_6=\sum_{\delta_1,\delta_2\in\{0,1\}}&\left(\begin{array}{ccccccccccc}  1 &  -1 & -1&-1& 0&0&0& 0&1&-1\\ 1&2-3\delta_1 &3-3\delta_2&\hspace{8pt} -1\hspace{8pt}& \hspace{11.3pt} 0\hspace{11.3pt}&\hspace{11.3pt} 0\hspace{11.3pt}&\hspace{8pt} -1\hspace{8pt}&\hspace{11.3pt} 0\hspace{11.3pt}&\hspace{8pt} -1\hspace{8pt}& 0 \end{array} \right),\\
A_7=\sum_{\delta_1,\delta_2,...,\delta_5\in\{0,1\}}&\left(\begin{array}{ccccccccccc}  1 &  -1 & -1&-1& -1&-1&-1& 0&1&-1\\ 1&2-3\delta_1 &3-3\delta_2&2-3\delta_3& 3-3\delta_4&3-3\delta_5&\hspace{8pt} -1\hspace{8pt}&\hspace{11.3pt} 0\hspace{11.3pt}&\hspace{8pt} -1\hspace{8pt}& 0 \end{array} \right),\\
A_8=\sum_{\delta_4,\delta_5\in\{0,1\}}&\left(\begin{array}{ccccccccccc}  1 &  -1 & 0&1& -1&-1&-1& 0&1&-1\\ 1&\hspace{8pt} -1\hspace{8pt} &\hspace{11.3pt} 0\hspace{11.3pt}&\hspace{8pt} -1\hspace{8pt}& 3-3\delta_4&3-3\delta_5&\hspace{8pt} -1\hspace{8pt}&\hspace{11.3pt} 0\hspace{11.3pt}&\hspace{8pt} -1\hspace{8pt}& 0 \end{array} \right).\end{array}$$
}
\normalsize{Then $x_4=\sum_{i=1}^8 A_i$.}
\end{exmp}

\section{Proofs}

 We need more notation.

\begin{defn}
For integers $u,n$ with $3\leq u\leq n-1$, let $$\mathcal{T}^{\geq u}(\mathcal{D}_n):=\left\{\{\beta_1,\cdots,\beta_t\}\,\left|\aligned&\bullet \,\,t\geq 1,\,\beta_j\in \mathcal{P}(\mathcal{D}_n)\text{ for all }1\leq j\leq t,\,\\
 &\bullet \text{ if }j\neq j'\text{ then }\beta_j\text{ and }\beta_{j'}\text{ have no common edge,}\\ 
 &\bullet \text{ if }\beta_j=\alpha(i,k)\text{ and }\beta_{j'}=\alpha(i',k')\text{ then }i\neq k'\text{ and }i'\neq k,\\
&\bullet \text{ and there exist integers }j,w,m,\text{ with }m\geq u\text{ such that }\\
&\,\,\,\,\,\,\,\,\,\,\,\,\,\,\beta_j\text{ is }(m,w)\text{-green and none of the }(c_{m-1}-wc_{m-2})\\&\,\,\,\,\,\,\,\,\,\,\,\,\,\,\text{preceding edges of }v_i\text{ is contained in any }\beta_{j'}\endaligned\right.\right\}.
$$
\end{defn}
\begin{defn}
Let $$ \widetilde{\mathcal{F}}(\mathcal{D}_n):=\left\{\{\beta_1,\cdots,\beta_t\}\,\left|\aligned&\bullet \,\,t\geq 0,\,\beta_j\in \mathcal{P}(\mathcal{D}_n)\text{ for all }1\leq j\leq t,\,\\ &\bullet \text{ if }j\neq j'\text{ then }\beta_j\text{ and }\beta_{j'}\text{ have no common edge,}\\ &\bullet \text{ and if }\beta_j=\alpha(i,k)\text{ and }\beta_{j'}=\alpha(i',k')\text{ then }i\neq k'\text{ and }i'\neq k\endaligned\right.\right\}.  $$  
\end{defn}

Note that \begin{equation}\label{20110516eq}{\mathcal{F}}(\mathcal{D}_n)=\widetilde{\mathcal{F}}(\mathcal{D}_n)\setminus\mathcal{T}^{\geq 3}(\mathcal{D}_{n}).\end{equation}

\begin{lem}\label{20110516eq2}
If $m\geq n-1$, then there do not exist $i,w$  $(1\leq w< r-1)$ such that $\min\{t\,|\,i<t\leq c_{n-2},\,s_{i,t}> s\}$ is of the form $i+c_m-wc_{m-1}$. In particular, for any $n\geq 4$, the set $\mathcal{T}^{\geq n-1}(\mathcal{D}_{n})$ is empty.
 \end{lem}
\begin{proof}
If $m\geq n-1$ and $\min\{t\,|\,i<t\leq c_{n-2},\,s_{i,t}> s\}=i+c_m-wc_{m-1}$, then $\min\{t\,|\,i<t\leq c_{n-2},\,s_{i,t}> s\}\geq c_{n-1}-wc_{n-2}$, which would be greater than $c_{n-2}$ because $w\leq r-2$. But this is a contradiction, because $v_{c_{n-2}}$ is the highest vertex in $\mathcal{D}_{n}$.
\end{proof}

Let $z_2=x_2$ and
\begin{equation}\label{20110411z}z_{n-1}=\sum_{\mathbf{\beta}\in\widetilde{\mathcal{F}}(\mathcal{D}_n)}xyx^{-1}y^{-1}x\left(\prod_{i=1}^{c_{n-1}}\beta_{[i]}\right)x^{-1} \end{equation}for $n\geq 4$.

\begin{lem}\label{20110411lem1} Let $n\geq 3$. Then
$$z_{n}=F(z_{n-1})+\sum_{\mathbf{\beta}\in\mathcal{T}^{\geq 3}(\mathcal{D}_{n+1})\setminus \mathcal{T}^{\geq 4}(\mathcal{D}_{n+1})}xyx^{-1}y^{-1}x\left(\prod_{i=1}^{c_{n}}\beta_{[i]}\right)x^{-1}.$$
\end{lem}

\begin{lem}\label{20110411lem2} Let $u\geq 3$ and $n\geq u+2$. Then
$$\aligned 
& F\left(\sum_{\mathbf{\beta}\in\mathcal{T}^{\geq u}(\mathcal{D}_n)\setminus \mathcal{T}^{\geq u+1}(\mathcal{D}_n)}xyx^{-1}y^{-1}x\left(\prod_{i=1}^{c_{n-1}}\beta_{[i]}\right)x^{-1}\right)\\
&=\sum_{\mathbf{\beta}\in\mathcal{T}^{\geq u+1}(\mathcal{D}_{n+1})\setminus \mathcal{T}^{\geq u+2}(\mathcal{D}_{n+1})}xyx^{-1}y^{-1}x\left(\prod_{i=1}^{c_{n}}\beta_{[i]}\right)x^{-1}.\endaligned$$
\end{lem}

\begin{lem}\label{20110411lem3} Let $n\geq 4$.  Then
\begin{equation}\label{20110411lem3f}\aligned  x_{n-1}&=z_{n-1}-\sum_{m=5}^n F^{n-m}\left(\sum_{\mathbf{\beta}\in\mathcal{T}^{\geq 3}(\mathcal{D}_m)\setminus \mathcal{T}^{\geq 4}(\mathcal{D}_m)}xyx^{-1}y^{-1}x\left(\prod_{i=1}^{c_{m-1}}\beta_{[i]}\right)x^{-1}\right)\\
&=\sum_{\mathbf{\beta}\in\mathcal{F}(\mathcal{D}_n)}xyx^{-1}y^{-1}x\left(\prod_{i=1}^{c_{n-1}}\beta_{[i]}\right)x^{-1}.
\endaligned\end{equation}
\end{lem}

The proof of Lemma~\ref{20110411lem1} will be independent of those of Lemmas~\ref{20110411lem2} and \ref{20110411lem3}. We prove Lemmas~\ref{20110411lem2} and \ref{20110411lem3} by the following induction: 
\begin{equation}\label{globalinduction}
\aligned&[\text{Lemma~\ref{20110411lem2} holds true for }n\leq d] \Longrightarrow \text{[Lemma~\ref{20110411lem3} holds true for }n\leq d+1]\\
& \Longrightarrow [\text{Lemma~\ref{20110411lem2} holds true for }n\leq d+1] \Longrightarrow [\text{Lemma~\ref{20110411lem3} holds true for }n\leq d+2] \cdots.
\endaligned\end{equation}

\begin{proof}[Proof of Lemma~\ref{20110411lem3}]
This proof is a straightforward adaptation of the proof of Lemma 19 in \cite{LS}.
We use induction on $n$. It is easy to show that $x_3=z_3$. Assume that (\ref{20110411lem3f}) holds for $n$.

Then $$\aligned 
x_{n}& =F(x_{n-1})\\
&\overset{F : homomorphism}= F(z_{n-1})-\sum_{m=5}^n F^{n-m+1}\left(\sum_{\mathbf{\beta}\in\mathcal{T}^{\geq 3}(\mathcal{D}_m)\setminus \mathcal{T}^{\geq 4}(\mathcal{D}_m)}xyx^{-1}y^{-1}x\left(\prod_{i=1}^{c_{m-1}}\beta_{[i]}\right)x^{-1}\right)\\
&\overset{Lemma~\ref{20110411lem1}}= z_{n}-\sum_{m=5}^{n+1} F^{n-m+1}\left(\sum_{\mathbf{\beta}\in\mathcal{T}^{\geq 3}(\mathcal{D}_m)\setminus \mathcal{T}^{\geq 4}(\mathcal{D}_m)}xyx^{-1}y^{-1}x\left(\prod_{i=1}^{c_{m-1}}\beta_{[i]}\right)x^{-1}\right)\\
&\overset{Lemma~\ref{20110411lem2}}= z_{n}-\sum_{m=5}^{n+1} \sum_{\mathbf{\beta}\in\mathcal{T}^{\geq n-m+4}(\mathcal{D}_{n+1})\setminus \mathcal{T}^{\geq n-m+5}(\mathcal{D}_{n+1})}xyx^{-1}y^{-1}x\left(\prod_{i=1}^{c_{n}}\beta_{[i]}\right)x^{-1}\\
&= z_{n}-\sum_{\mathbf{\beta}\in\mathcal{T}^{\geq 3}(\mathcal{D}_{n+1})\setminus \mathcal{T}^{\geq n}(\mathcal{D}_{n+1})}xyx^{-1}y^{-1}x\left(\prod_{i=1}^{c_{n}}\beta_{[i]}\right)x^{-1}\\
&\overset{Lemma~\ref{20110516eq2}}= z_{n}-\sum_{\mathbf{\beta}\in\mathcal{T}^{\geq 3}(\mathcal{D}_{n+1})}xyx^{-1}y^{-1}x\left(\prod_{i=1}^{c_{n}}\beta_{[i]}\right)x^{-1}\\
&\overset{(\ref{20110411z})}= \sum_{\mathbf{\beta}\in\widetilde{\mathcal{F}}(\mathcal{D}_{n+1})\setminus\mathcal{T}^{\geq 3}(\mathcal{D}_{n+1})}xyx^{-1}y^{-1}x\left(\prod_{i=1}^{c_{n}}\beta_{[i]}\right)x^{-1}\\
&\overset{(\ref{20110516eq})}=\sum_{\mathbf{\beta}\in{\mathcal{F}}(\mathcal{D}_{n+1})}xyx^{-1}y^{-1}x\left(\prod_{i=1}^{c_{n}}\beta_{[i]}\right)x^{-1}.\endaligned$$
\end{proof}

In order to prove Lemma~\ref{20110411lem1}, we need the following notation. 

\begin{defn}
The sequence $\{b_{i,j}\}_{i\in \mathbb{Z}_{\geq 2}, 1\leq j\leq c_i}$ is defined by:
$$b_{i,j}=\left\{\begin{array}{ll} r, &\text{ if }\alpha_j\text{ is a horizontal edge of }\mathcal{D}_{i+1} \\
r-1,&\text{ if }\alpha_j\text{ is a vertical edge of }\mathcal{D}_{i+1}.    \end{array}   \right.$$
\qed\end{defn}

For integers $i\leq j$, we denote the set $\{i,i+1,i+2,\cdots,j\}$ by $[i,j]$. We will always identify $[i,j]$ with the subpath given by $(\alpha_i,\alpha_{i+1},\cdots,\alpha_j)$. 

\begin{defn}\label{def_of_f}
We will need a function $f$ from $\{\text{subsets of } [1,c_{n-1}]\}$ to $\{\text{subsets of } [1,c_{n}]\}$. For each subset $V\subset [1,c_{n-1}]$, we define $f(V)$ as follows.

If $V=\emptyset$ then $f(\emptyset)=\emptyset$. If $V\neq\emptyset$ then we write $V$ as a disjoint union of maximal connected subsets  $ V=\sqcup_{i=1}^j [e_i,e_i+\ell_i-1]$ with $\ell_i>0$ $(1\leq i\leq j)$ and $e_i+\ell_i<e_{i+1}$ $(1\leq i\leq j-1)$. For each $1\leq i\leq j$, let $$W_i=[1+\sum_{k=1}^{e_i-1}b_{n-1,k},\, \sum_{k=1}^{e_i+\ell_i-1}b_{n-1,k}]$$ and define $f_i(V)$ by
$$
f_i(V):=\left\{ \begin{array}{ll}  
W_i,  &\begin{array}{l}\text{if the subpath given by }W_i\text{ is blue or green;} \end{array}\\
 \text{ } & \text{ }\\
\{\sum_{k=1}^{e_i-1}b_{n-1,k}\}\cup W_i, & \text{ otherwise.} \end{array}  \right.
$$
Then $f(V)$ is obtained by taking the union of $f_i(V)$'s:
$$
f(V):=\cup_{i=1}^j f_i(V).
$$Note that the subpath given by $f_i(V)$ is always one of blue, green, or red subpaths, and that every blue, green, or red subpath can be realized as the image of a maximal connected interval under $f$.
\qed\end{defn}

\begin{exmp}
Let $r=3$ and $n=4$. Then $f(\{1,2,3\})=\{1,2,3,4,5,6,7,8\}.$ As illustrated below, the image of the subpath $(\alpha_1,\alpha_2,\alpha_3)$ under $f$ is the subpath $(\alpha_1,\cdots,\alpha_{8}),$ which is blue.

$$
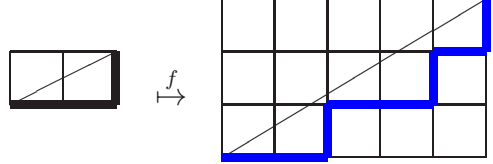
\begin{picture}(80,80)
\boxs{0,60}\boxs{20,60} \linethickness{1pt}\put(0,60){\line(2,1){40}}\linethickness{3pt}\put(0,60){\line(1,0){40}}
\linethickness{3pt}\put(40,60){\line(0,1){20}}
\put(55,60){$\Huge{\overset{f}\mapsto}$}\end{picture}
\begin{picture}(140,80)
\boxs{0,60}\boxs{20,60}\boxs{40,60}\boxs{60,60}\boxs{80,60}
\boxs{0,80}\boxs{20,80}\boxs{40,80}\boxs{60,80}\boxs{80,80}
\boxs{0,40}\boxs{20,40}\boxs{40,40}\boxs{60,40}\boxs{80,40}
\linethickness{1pt}\put(0,40){\line(5,3){100}}
\linethickness{3pt}\color{blue}\put(0,40){\line(1,0){40}}
\linethickness{3pt}\put(40,40){\line(0,1){20}}
\linethickness{3pt}\put(40,60){\line(1,0){40}}
\linethickness{3pt}\put(80,60){\line(0,1){20}}
\linethickness{3pt}\put(80,80){\line(1,0){20}}
\linethickness{3pt}\put(100,80){\line(0,1){20}}
\end{picture} $$
\end{exmp}

\begin{lem}
Let $C=xyx^{-1}y^{-1}$. Then $F(C)=C$.
\end{lem}
\begin{proof}
$$F(xyx^{-1}y^{-1})=xyx^{-1}(1+y^r)x^{-1}xy^{-1}x^{-1}x(1+y^r)^{-1}=xyx^{-1}y^{-1}.$$
\end{proof} 

\begin{proof}[Proof of Lemma~\ref{20110411lem1}]
The idea is the same as in the commutative case  \cite[Lemma 17]{LS}, that is, we choose any subset, say $V$, of $\{\alpha_1,\cdots,\alpha_{c_{n-1}}\}$ and consider all $\beta$ whose support is $V$. Then one can check that  $$\aligned 
& F\left(\sum_{\begin{array}{c}\beta\in {\mathcal{F}}(\mathcal{D}_{n}),\\\beta : \text{supp}(\beta)=V\end{array} } Cx\left(\prod_{i=1}^{c_{n-1}}\beta_{[i]}\right)x^{-1}\right)\\
& =\sum_{\begin{array}{c}\beta\in \widetilde{\mathcal{F}}(\mathcal{D}_{n+1}),\\ \beta: \text{colored subpaths of }\beta\text{ are precisely the ones given by }f(V)\end{array}} Cx\left(\prod_{i=1}^{c_{n}}\beta_{[i]}\right)x^{-1} \endaligned$$  Then, as in the commutative case, $\beta\in \widetilde{\mathcal{F}}(\mathcal{D}_{n+1})$ which are not covered by this construction belong to $\mathcal{T}^{\geq 3}(\mathcal{D}_{n+1})\setminus \mathcal{T}^{\geq 4}(\mathcal{D}_{n+1})$.

For example, if $r=3$, $n=4$, and $V=\emptyset$, then $\beta $ must be empty, and we get $$\prod_{i=1}^{3}\beta_{[i]} = x^{-1}y^3 x^{-1}y^3 x^{-1}y^2.$$ It is straightforward to show that 
$$\aligned
&F\left(Cx\left(x^{-1}y^3 x^{-1}y^3 x^{-1}y^2\right)x^{-1}\right)\\
&=\sum_{\beta\subset\{\alpha_1,\cdots,\alpha_{8}\}  } Cx\left(\prod_{i=1}^{8}\beta_{[i]}\right)x^{-1}\\
&=A_1,
\endaligned$$
where $A_1$ is the same one as defined in Example~\ref{mainexmp}.

If $r=3$, $n=4$, and  $V=\{1,2,3\}$, then $\beta $ is either $\alpha(0,1)$ or $\{\alpha_1,\alpha_2,\alpha_3\}$, and  we get $\prod_{i=1}^{3}\beta_{[i]} = x^{-1}y^0 x^{-1}y^0 x^{-1}y^{-1} $
 or $\ \prod_{i=1}^{3}\beta_{[i]} = x^{0}y^0 x^{0}y^0 x^{0}y^{-1}$. Then
$$\aligned
&F\left(Cx\left(x^{-1}y^0 x^{-1}y^0 x^{-1}y^{-1} + x^{0}y^0 x^{0}y^0 x^{0}y^{-1}\right)x^{-1}\right)\\
&=Cxyx^{-1}\left(xy^{-3}x^{-1} x (1+y^3)^{-1} + x (1+y^3)^{-1}\right)xy^{-1}x^{-1}\\
&=Cxy\left(y^{-3}(1+y^3)^{-1} +  (1+y^3)^{-1}\right)xy^{-1}x^{-1}\\
&=Cxy\left(y^{-3}(1+y^3)^{-1} +y^{-3}y^3  (1+y^3)^{-1}\right)xy^{-1}x^{-1}\\
&=Cxy\left(y^{-3}(1+y^3)(1+y^3)^{-1} \right)xy^{-1}x^{-1}\\
&=Cxy\left(y^{-3} \right)xy^{-1}x^{-1}\\
&=\sum_{\beta=\{\alpha(0,3)\}  } Cx\left(\prod_{i=1}^{c_{n}}\beta_{[i]}\right)x^{-1}\\
&=A_4,
\endaligned$$
where $A_4$ is the same one as defined in Example~\ref{mainexmp}.

 \end{proof}

It remains to prove Lemma~\ref{20110411lem2}.

\begin{proof}[Sketch of Proof of Lemma~\ref{20110411lem2}]
This proof is a straightforward adaptation of the proof of Lemma 18 in \cite{LS}.

Here we will deal only with the case of $n=u+2$. The case of $n>u+2$ makes use of the same argument. As we use the induction (\ref{globalinduction}), we can assume that $$x_j=\sum_{\mathbf{\beta}\in\mathcal{F}(\mathcal{D}_i)}xyx^{-1}y^{-1}x\left(\prod_{i=1}^{c_{j-1}}\beta_{[i]}\right)x^{-1}$$for $j\leq n$. 

Since it is straightforward to check the statement for $n=5$, we assume that $n\geq 6$. For any $w\in [1,r-2]$, it is easy to show that the lattice point
$(w(c_{n-2}-c_{n-3}), wc_{n-3})$ is below the diagonal from $(0,0)$ to $(c_{n-1}-c_{n-2},c_{n-2})$ and that the points $(w(c_{n-2}-c_{n-3}), 1+wc_{n-3})$ and $(w(c_{n-2}-c_{n-3})-1, wc_{n-3})$ are above the diagonal. So $(w(c_{n-2}-c_{n-3}), wc_{n-3})$ is one of the vertices $v_i$ on $\mathcal{D}_n$. Actually $v_{wc_{n-3}}=(w(c_{n-2}-c_{n-3}), wc_{n-3})$. Since $u=n-2$ and $\alpha(wc_{n-3},c_{n-2})$ is the only $(n-2,w)$-green subpath in $\{\alpha(i,k)\,|\, 0 \leq i < k \leq c_{n-2}\}$, every $\mathbf{\beta}\in\mathcal{T}^{\geq u}(\mathcal{D}_n)\setminus \mathcal{T}^{\geq u+1}(\mathcal{D}_n)$ must contain the green subpath from $v_{wc_{n-3}}$ to $v_{c_{n-2}}$. Then none of the 
$c_{n-3}-wc_{n-4}$ preceding edges of $v_{wc_{n-3}}$ is contained in any element $\beta_{j'}$ of $\beta$. The green subpath from $v_{wc_{n-3}}$ to $v_{c_{n-2}}$ corresponds to the interval $[wc_{n-2}+1,c_{n-1}] \subset [1,c_{n-1}]$. The 
$c_{n-3}-wc_{n-4}$ preceding edges of $v_{wc_{n-3}}$ are $\alpha_{(rw-1)c_{n-3}+1},\cdots, \alpha_{wc_{n-2}}$.

Thus we have
$$\aligned 
&\sum_{\mathbf{\beta}\in\mathcal{T}^{\geq u}(\mathcal{D}_n)\setminus \mathcal{T}^{\geq u+1}(\mathcal{D}_n)}xyx^{-1}y^{-1}x\left(\prod_{i=1}^{c_{n-1}}\beta_{[i]}\right)x^{-1}\\
&=\sum_{w=1}^{r-2} \sum_{V\subset [1,(rw-1)c_{n-3}]}\,\,\sum_{\mathbf{\beta}: \cup\beta_i=V\cup[wc_{n-2}+1,c_{n-1}],\,\mathbf{\beta}\ni \alpha(wc_{n-3},c_{n-2})}xyx^{-1}y^{-1}x\left(\prod_{i=1}^{c_{n-1}}\beta_{[i]}\right)x^{-1}.\,\,\,\,\,\,\,\,(*)\endaligned$$

We observe that the subpath corresponding to $[1,(rw-1)c_{n-3}]$ consists of $(w-1)$ copies of $\mathcal{D}_{n-1}$, $(r-1)$ copies of $\mathcal{D}_{n-2}$, and $(w-1)$ copies of $\mathcal{D}_{n-3}$. Let $v_{j_0}=(0,0)$ and $v_{j_i}$ be the end point of each of these copies, i.e., $$\aligned &v_{j_i}=v_{ic_{n-3}}\text{ for }1\leq i\leq w-1,\\
& v_{j_{w-1+i}}=v_{(w-1)c_{n-3}+ic_{n-4}}\text{ for }1\leq i\leq r-1,\\
& v_{j_{w+r-2+i}}=v_{(w-1)c_{n-3}+(r-1)c_{n-4}+ic_{n-5}}\text{ for }1\leq i\leq w-1.\endaligned$$
If a $(m,w')$-green (resp. blue or red) subpath, say $\alpha(i,k)$, in $[1,(rw-1)c_{n-3}]$ passes through $v_{j_e},v_{j_{e+1}},\cdots, v_{j_{e+\ell}}$, then $\alpha(i,k)$ can be naturally decomposed into $\alpha(i,j_e)$, $\alpha(j_e,j_{e+1})$, $\cdots$,  $\alpha(j_{e+\ell},k)$. It is not hard to show that $\alpha(i,j_e)$ is also $(m,w')$-green (resp. blue or red) and that 
 $\alpha(j_e,j_{e+1})$, $\cdots$,  $\alpha(j_{e+\ell},k)$ are all blue.

Hence 
$$\aligned
&(*)=\sum_{w=1}^{r-2} xyx^{-1}y^{-1}x\left(\sum_{\mathbf{\beta}\in\mathcal{F}(\mathcal{D}_{n-1})}\left(\prod_{i=1}^{c_{n-2}}\beta_{[i]}\right)\right)^{w-1}\left(\sum_{\mathbf{\beta}\in\mathcal{F}(\mathcal{D}_{n-2})}\left(\prod_{i=1}^{c_{n-3}}\beta_{[i]}\right)\right)^{r-1}\\
&\,\,\,\,\,\,\,\,\,\,\,\,\,\,\,\,\times\left(\sum_{\mathbf{\beta}\in\mathcal{F}(\mathcal{D}_{n-3})}\left(\prod_{i=1}^{c_{n-4}}\beta_{[i]}\right)\right)^{w-1}x^{-1}yxy^{-1} x^{-1}\\
&=\sum_{w=1}^{r-2} xyx^{-1}y^{-1}x\left(x^{-1}yxy^{-1}x^{-1}x_{n-2}x\right)^{w-1}\left(x^{-1}yxy^{-1}x^{-1}x_{n-3}x\right)^{r-1}\\
&\,\,\,\,\,\,\,\,\,\,\,\,\,\,\,\,\times\left(x^{-1}yxy^{-1}x^{-1}x_{n-4}x\right)^{w-1}x^{-1}yxy^{-1}x^{-1}\\
&=\sum_{w=1}^{r-2} C(C^{-1}x_{n-2})^{w-1}(C^{-1}x_{n-3})^{r-1}(C^{-1}x_{n-4})^{w-1}C^{-1}.
\endaligned$$For the same reason, we get $$\sum_{\mathbf{\beta}\in\mathcal{T}^{\geq u+1}(\mathcal{D}_{n+1})\setminus \mathcal{T}^{\geq u+2}(\mathcal{D}_{n+1})}Cx\left(\prod_{i=1}^{c_{n}}\beta_{[i]}\right)x^{-1}=\sum_{w=1}^{r-2} C(C^{-1}x_{n-1})^{w-1}(C^{-1}x_{n-2})^{r-1}(C^{-1}x_{n-3})^{w-1}C^{-1}.$$
Since $F(C)=C$, we have
$$\aligned 
& F\left(\sum_{\mathbf{\beta}\in\mathcal{T}^{\geq u}(\mathcal{D}_n)\setminus \mathcal{T}^{\geq u+1}(\mathcal{D}_n)}xyx^{-1}y^{-1}x\left(\prod_{i=1}^{c_{n-1}}\beta_{[i]}\right)x^{-1}\right)\\
&=\sum_{\mathbf{\beta}\in\mathcal{T}^{\geq u+1}(\mathcal{D}_{n+1})\setminus \mathcal{T}^{\geq u+2}(\mathcal{D}_{n+1})}xyx^{-1}y^{-1}x\left(\prod_{i=1}^{c_{n}}\beta_{[i]}\right)x^{-1}.\endaligned$$
 \end{proof}

\end{document}